\documentclass[10pt, conference]{ieeeconf}
\IEEEoverridecommandlockouts
\usepackage{srcltx,pdfsync}
\usepackage{amsmath}
\usepackage{amsfonts}
\usepackage{amssymb}
\usepackage{graphicx,subfigure,psfrag}
\usepackage{url,ifthen}
\usepackage[dvipsnames,usenames]{color}
\usepackage[colorlinks, linkcolor=Blue, filecolor =Red, citecolor=RoyalPurple,]{hyperref}

\usepackage[fancythm]{jphmacros2e}
 
\makeatletter
\let\NAT@parse\undefined
\makeatother
\usepackage[numbers,sort&compress]{natbib}

\newcommand{\xplus}{x^{k+1}_i}

\newcommand{\tr}{^{\sf T}}

\newcommand{\jb}[1]{\footnote{{\bf JB's comment: #1}}}
\renewcommand{\jb}[1]{} 

\newlength{\noteWidth}
\long\def\notes#1{\ifinner
           {\footnotesize #1}
           \else
           \marginpar{\parbox[t]{\noteWidth}{\raggedright\footnotesize #1}}
       \fi\typeout{#1}}

\begin{document}
\title{A Decentralized Multi-block ADMM for Demand-side Primary Frequency Control using Local Frequency Measurements\thanks{$^*$Author names are ordered alphabetically.}}

\author{Jonathan Brooks$^*$, William Hager$^*$, and Jiajie Zhu$^*$}

\maketitle

\begin{abstract}
We consider demand-side primary frequency control in the power grid provided by smart and flexible loads: loads change consumption to match generation and help the grid while minimizing disutility for consumers incurred by consumption changes. The dual formulation of this problem has been solved previously by Zhao \emph{et al.} in a decentralized manner for consumer disutilities that are twice continuously differentiable with respect to consumption changes. In this work, we propose a decentralized multi-block alternating-direction-method-of-multipliers (DM-ADMM) algorithm to solve this problem. In contrast to the ``dual algorithm'' of Zhao \emph{et al.}, the proposed DM-ADMM algorithm does not require the disutilities to be continuously differentiable; this allows disutility functions that model consumer behavior that may be quite common. In this work, we prove convergence of the DM-ADMM algorithm in the deterministic setting (i.e., when loads may estimate the consumption-generation mismatch from frequency measurements exactly). We test the performance of the DM-ADMM algorithm in simulations, and we compare (when applicable) with the previously proposed solution for the dual formulation. We also present numerical results for a previously proposed ADMM algorithm, whose results were not previously reported.
\end{abstract}

\section{Introduction}
\label{sec:intro}
In the power grid, generation must match consumption. If there is too large of a mismatch, the generators may be shut down to prevent damage, and blackouts may occur as a result~\cite{kirby2007ancillary}. When a contingency occurs, and there is a sudden change in balance between generation and consumption (e.g., if generation from a renewable-energy source suddenly goes off-line), the balance must be quickly restored. Traditionally, this restoration is achieved by fast-ramping, reserve generators that are either already on-line or standing by. These generators are often fossil-fuel generators, and this ramping can increase total emissions---counteracting the ``green'' effects of the renewable-energy generators~\cite{lew2013western}. With the increasing penetration of volatile renewable energy sources, additional resources are required to ensure its stable and reliable operation.

Loads offer another such resource. Instead of generators changing generation, some loads may change consumption---providing the same service to the grid but without the increased emissions from ramping generators~\cite{siano2014demand,milligan2010utilizing}. Such loads are a powerful resource to address the imbalance in the grid. For example, commercial buildings' heating, ventilation, and air-conditioning (HVAC) systems account for approximately 6\% of the total energy consumption in the United States~\cite{Energy_EIA:2015}. However, providing this service to the grid may incur some disutility for end users (e.g., reduced ventilation or cooling in a commercial building). Thus a balance must be found between providing valuable service to the grid while minimizing consumers' disutility; this can be cast as a standard optimization problem where disutility is the objective function and the consumption-generation balance is an equality constraint.

Due to the size of the grid, centralized solutions are not practical. Therefore, decentralization is paramount and has been the focus of much of the literature. 
In particular, loads may use local power-frequency measurements to infer information about the grid as a whole~\cite{schFAPER80}. Using such measurements, loads can provide primary frequency control~\cite{NERCBalancingDocument:2011,pnnl2008value,short2007stabilization,molina2011decentralized} in a decentralized manner, which is the subject of this paper.

This work is inspired by the recent work of Zhao \emph{et al.} in~\cite{zhao2013optimal}. Zhao \emph{et al.} proposed an algorithm that solves the dual formulation of the load control problem in a decentralized manner. Loads use local frequency measurements to estimate the mismatch between consumption and generation (i.e., loads estimate the amount of violation of the equality constraint). These estimates are then used for dual ascent. Loads may also communicate with other loads to average out noise in the estimates.

The algorithm proposed by Zhao \emph{et al.} requires the consumer disutilities to be twice continuously differentiable with respect to consumption changes. However, for some consumers, it may be more costly to decrease consumption than to increase consumption (or vice versa). Such a model for consumer behavior may not be twice continuously differentiable (e.g., $f_1$ in Figure~\ref{fig:f}). We hypothesize that such disutilities may be quite common; for example, some balancing authorities have separate markets for providing upward and downward regulation services~\cite{kirby2010providing}. This allows ancillary-service providers to judge the disutility of increasing or decreasing output independently. For example, separate payment rates for upward and downward service by a commercial HVAC system are used in a contractual framework proposed in~\cite{maasoumy2014model}. Additionally, consumers may be able to change consumption without incurring much disutility if the changes are small, but there may be significant disutility for larger changes. For example, it was shown in~\cite{linbarmeymid:2015} that commercial HVAC systems can be used to provide regulation without adverse effects on indoor climate so long as the changes in consumption are small and band-limited. Similarly, the authors of~\cite{todd2008providing} studied an aluminum-smelting plant's ability to provide ancillary services; the results of the study suggested that there may be little disutility if consumption changes are small, but there is significant disutility if changes are too large or last for too long. The loads described in~\cite{linbarmeymid:2015,todd2008providing} might have a disutility function similar to $f_2$ in Figure~\ref{fig:f}, which is not continuously differentiable.
 

\begin{figure}[htpb]
\includegraphics[scale=0.3]{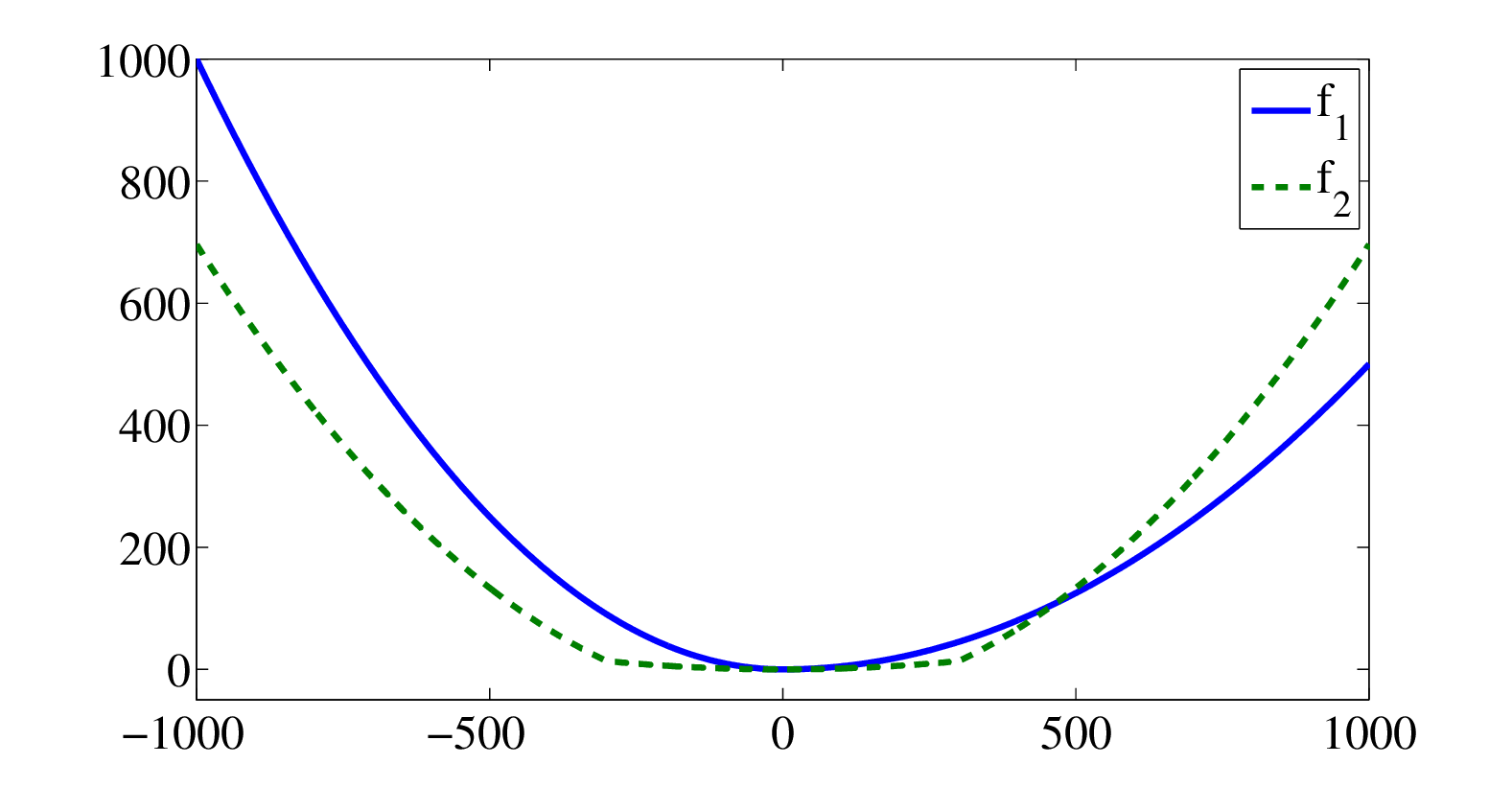}
\caption{Models of consumer disutility that are not twice continuously differentiable with respect to change in consumption. Both functions are strongly convex, but $f_1$ is not twice continuously differentiable, and $f_2$ is not continuously differentiable.}
\label{fig:f}
\end{figure}

In this work, we propose an alternative solution to the load control problem that requires fewer restrictions on the consumer disutility functions. Specifically, our proposed algorithm does not require consumer disutilities to be continuously differentiable. This is a significant extension to the alogrithm of Zhao \emph{et al.} because it allows disutility functions that model consumer behaviors such as those shown in Figure~\ref{fig:f}. Our proposed solution is a decentralized version of the alternating-direction method of multipliers (ADMM), where loads use local frequency measurements to estimate the violation of the equality constraint. The use of decentralized ADMM for smart-grid applications was discussed in~\cite{liu2015multi}. The authors of that work suggested using a ``proximal Jacobian'' ADMM (PJ-ADMM) algorithm, whose convergence was proven in~\cite{deng2013parallel}, but no numerical results were reported. 

In this preliminary work, we prove the convergence of our algorithm to a feasible-optimal solution (i.e., one that minimizes disutility while maintaining the equality constraint). Our convergence proof is limited to the deterministic setting, in which loads may estimate the consumption-generation mismatch exactly from frequency measurements. However, we test our algorithm through simulations with noisy estimates, and results indicate our algorithm performs well even in the presence of noise. We compare the performance of our proposed algorithm with that of the algorithm proposed by Zhao \emph{et al.} (when applicable), and we find that performance is comparable. Additionally, we present numerical results for the PJ-ADMM algorithm, and we find similar performance compared to the DM-ADMM algorithm as well.

This paper is organized as follows. We formally introduce the problem in Section~\ref{sec:prob}. We introduce our proposed solution in Section~\ref{sec:dm-admm}. Section~\ref{sec:converge} details our proof of convergence of the proposed algorithm. In Section~\ref{sec:sim}, we describe our simulation environment and present the results of the simulations. Finally, Section~\ref{sec:conc} concludes this work and provides possible avenues for future work.

\section{Problem Formulation}\label{sec:prob}
We consider the problem from~\cite{zhao2013optimal}, in which there is a microgrid with one generator and $n$ flexible loads with one system-wide frequency. We divide time into a discrete time axis $k=0,1,\ldots$. Denote the generation output at time $k$ by $g^k$, and denote its nominal value by $g^0$. In this paper, we consider a sudden change in generation output, which we model as a step change of size $C$. It is the goal of the flexible loads to change their consumption in order to maintain system frequency at the nominal value $\omega^0$; we denote this change in consumption by $x_i$ for load $i$. In addition, load $i$ is constrained in how much it may change its consumption; i.e., $x_i\in[a_i,~b_i]$. At the same time, load $i$ incurs a disutility $f_i(x_i)$ by changing consumption, and the loads must minimize the total disutility $\sum_if_i(x_i)$. More formally, we consider the optimization problem
\begin{equation}
\label{knapsack}
\begin{aligned}
& \underset{x_i,\ i=1,\ldots,n}{\text{minimize}}
& & \sum_{i=1}^{n}f_i (x_i) \\
& \text{subject to}
& & a_i \leq x_i \leq b_i, \; i = 1, \ldots, n,\\
& 
& & \sum_{i=1}^{n} x_i = C,
\end{aligned}
\end{equation}
where agents must make decisions in a decentralized manner.

One of the prominent appearances of the similar problems is in solving the dual formulation of support vector machine(SVM). Many optimization methods, such as \citep{Platt98,Joachims98,df03c,Zanni06,ghz11}, have been developed to solve SVM. If the objective function in problem \eqref{knapsack} is a separable quadratic function, the problem is refered to as the separable convex quadratic knapsack problem. A few algorithms, such as \citep{pardalos90} and \cite{davisefficient} have been proposed to solve it. All the aforementioned methods are efficient solvers of the optimization problem. However, all the methods mentioned are centralized algorithms. It means they rely on a centralized control unit to update and store variables.

The linear equality constraint in~\eqref{knapsack} requires information to be shared among agents. In the power grid, it is impractical for agents (loads) to share such information due to the possible number of agents as well as privacy concerns; e.g., load $i$ may not want to share information regarding its consumption or disutility with other loads. Hence, agent $i$ does not have the information of other functions $f_j$'s or variables $x_j$'s, for $j\neq i$.

However, agents have the ability to locally measure the system frequency, which can be used to estimate the consumption-generation mismatch~\cite{zhao2013optimal}. That is, the local frequency measurements allow each load to estimate the primal residual
\begin{align*}
r^k=\sum_{i=1}^n x_i^k-C,
\end{align*}
which is a measurement of feasibility. In practice, this estimate will be noisy; i.e., load $i$ will have access to
\begin{align*}
\hat{r}_i^k=r^k+e_i^k,
\end{align*}
where $e_i^k$ is the estimation error. It was shown in~\cite{zhao2013optimal} that, using the estimator described in~\cite{kitanidis1987unbiased}, $e_i^k$ is a martingale-difference sequence. This estimator is describedin Section~\ref{sec:est}.


\section{Decentralized Multi-block ADMM}\label{sec:dm-admm}
%
In the standard form, the ADMM splits the primal variable into two blocks $x$ and $z$, but it is desirable for the problem to 
be completely separable; i.e., each update may be done in a distributed fashion by each component. 
It is worth noting that~\citep{luohong2012,deng2013parallel} mentioned variants of ADMM using Jacobi update schemes. Their algorithms and convergence analysis are different from ours.

We first consider the following augmented formulation of Problem~\eqref{knapsack}:

\begin{equation}
\label{aug_knapsack}
\begin{aligned}
& \underset{x}{\text{minimize}}
& & \sum_{i=1}^{n}f_i (x_i) + \frac{\rho}{2}\|\sum_{i=1}^{n} x_i - C\|^2\\
& \text{subject to}
& & a_i \leq x_i \leq b_i, \; i = 1, \ldots, n,\\
& 
& & \sum_{i=1}^{n} x_i = C.
\end{aligned}
\end{equation}
With this formulation, we propose the Decentralized Multi-block ADMM (DM-ADMM) with the following update rule for agent $i$, which is an adaptation of the ADMM algorithm.
\bigskip

{\tt
\begin{tabular}{l}
{\sc Decentralized Multi-block ADMM}\\
\hline
Distributed task. For each block $i$,\\

\begin{minipage}{5cm}
\begin{equation}
\label{update_i}
\begin{aligned}
y_i^{k+1} & = y_i^{k} + \rho \hat{r}_i^k,\\
x_i^{k+1} & = \argmin_{a_i \leq x \leq b_i} f_i(x) + {y_i^{k+1}}(x+\hat{r}_i^k-x_i^k)\\
          &\quad\quad\quad\quad\quad\quad\quad+\frac{\rho}{2}\|x+\hat{r}_i^k-x_i^k\|^2.
\end{aligned}
\end{equation}
\end{minipage}\\
End.\\
\hline \\
\end{tabular}}




\section{Convergence of the DM-ADMM Algorithm}\label{sec:converge}
We will use the terminology by the foundational work of~\citep{Rockafellar70}. 

In this preliminary work, to simplify the convergence analysis, we consider the scenario without noise, and we leave analysis of the noised update for future work. We also assume that each agent $i$ has the same initial dual variable $y_i^0$. Finally, we make two general assumptions on the disutility functions $f_i$, which are easily satisfied because the form of $f_i$ is a modeling choice. These assumptions are summarized below.

\begin{assumption}\label{as:res}
$\hat{r}_i^k=r^k$ for all $i$ and $k$.
\end{assumption}
\begin{assumption}\label{as:y}
$y_i^0=y_j^0$, for all $i,\ j$.
\end{assumption}
\begin{assumption}\label{as:f}
The function $f_i$ is proper, lower semi-continuous, and convex for all $i$.
\end{assumption}
\begin{assumption}\label{as:strong}
The total disutility $\sum_i f_i$ is strongly convex.
\end{assumption}

In the proof below, we simplify the problem by dropping the box constraint $a_i\leq x \leq b_i$. Such a constraint can be incorporated into the proof by replacing the objective function $f_i(x)$ by $f_i(x)+g_i(x)$, where $g_i(x)$ is the indicator function of the interval $[a_i,b_i]$. The resulting objective function is still a proper, lower semi-continuous convex function. Thus the proof remains the same.

It is obvious that, in our update rule, all the $y_i^k$'s are equal if Assumptions~\ref{as:res} and~\ref{as:y} hold. We thus drop the sub-index for the dual variable in our proof.

If $r^{k+1}=0$, then the equality constraint of Problem~\eqref{knapsack} is satisfied. We use $p^\star = \sum_{i=1}^nf_i(x^\star_i)$ and $p^{k} = \sum_{i=1}^nf_i(x^{k}_i)$ to denote the global minimum value and the objective value at the $k$-th iteration.

\begin{proposition}
\label{obj_cvgs}
Let Assumptions~\ref{as:res},~\ref{as:y}, and~\ref{as:f} hold. Let $x^\star$ be the global minimizer of Problem~\eqref{knapsack} and $p^\star = \sum_{i=1}^nf_i(x^\star_i)$ be the global minimum value. Let $x^{k+1}$ and $y^{k+1}$ be the iterates generated by the update rule~\eqref{update_i} and $p^{k+1}=\sum_{i=1}^nf_i(\xplus)$ be the corresponding objective function value at that iteration.
Then
\begin{align}
\label{ppstar}
\begin{aligned}
p^{k+1} - p^\star \leq &-y^{k+2}r^{k+1} - \rho(r^k-r^{k+1})r^{k+1}\\
&\ \  - \rho\sum_{i=1}^n (x^k_i - x^{k+1}_i)(x_i^\star-\xplus)
\end{aligned}
\end{align}
\end{proposition}
\begin{proof}
See Appendix.
\end{proof}

\begin{lemma}
\label{lb_obj}
Let Assumptions~\ref{as:res},~\ref{as:y},~\ref{as:f}, and~\ref{as:strong} hold. Let $(x^\star, y^\star)$ be the global solution of Problem~\eqref{knapsack}, and let $p^\star$ and ${p}^{k+1}$ denote the objective function value at $x^\star$ and ${x}^{k+1}$, respectively. 
Then
\begin{equation}
\label{stronglyconvex}
p^\star - p^{k+1} < {y^\star} \big(\sum_i{x}^{k+1}_i -C\big) - \xi \|{x}^{k+1} - x^\star\|^2,
\end{equation}
where $\xi$ is a constant.
\end{lemma}
\begin{proof}
See Appendix.
\end{proof}

\begin{proposition}
\label{primal_res}
Let Assumptions~\ref{as:res},~\ref{as:y},~\ref{as:f}, and~\ref{as:strong} hold. Let $\left( x^{k}, y^k \right)$ be the pair generated by update rule~\eqref{update_i} and $(x^\star, y^\star)$ be the global solution of Problem~\eqref{knapsack}. 
If $\displaystyle\rho\leq \frac{\xi}{2\left(n-1\right)}$, where $\xi$ is the constant from Lemma~\ref{lb_obj}, then
\begin{align}
\label{rk_ineq}
\begin{aligned}
&\left(\frac{1}{\rho} \|y^{k+1}-y^\star\|^2 + (\rho+\xi)\| x^{k} - x^\star\|^2 \right)\\
&- \left(\frac{1}{\rho}\|y^{k+2}-y^\star\|^2 +(\rho+\xi) \|x^{k+1} - x^{\star}\|^2\right)\geq
\rho (r^k)^2,
\end{aligned}
\end{align}
where $\rho$ is the constant used in~\eqref{update_i}.
\end{proposition}
\begin{proof}
See Appendix.
\end{proof}

\begin{theorem}\label{the:main}
Let Assumptions~\ref{as:res},~\ref{as:y},~\ref{as:f}, and~\ref{as:strong} hold. Then the primal residual $r^k\to 0$ as $k\to\infty$. Furthermore, if there exists $\sigma >0$ such that $\displaystyle\rho\leq \frac{\xi-\sigma}{2\left(n-1\right)}$, then the primal variable converges $x^k\to x^\star$ and the objective function value converges $p^k\to p^\star$.
\end{theorem}
\begin{proof}
Summing the inequality~\eqref{rk_ineq} from $k=1$ to $k=K$ and letting $K\to\infty$, we have
\begin{equation}
\label{rk_zero}
\begin{aligned}
\frac{1}{\rho} \|y^{1}-y^\star\|^2 + (\rho+\xi)\| x^{0}-x^\star\|^2\geq
\sum_{k=1}^{\infty}\rho {r^k}^2.
\end{aligned}
\end{equation}
Therefore, $r^k\to 0$ as $k\to\infty$. 

If $\displaystyle\rho\leq \frac{\xi-\sigma}{2\left(n-1\right)}$, inequality~\eqref{jenzen} (see Appendix) can be re-written as
\begin{equation}
\label{jenzen2}
\begin{aligned}
&\left(\frac{1}{\rho} \|y^{k+1}-y^\star\|^2 + (\rho+\xi)\| x^{k} - x^\star\|^2 \right)\\
&- \left(\frac{1}{\rho}\|y^{k+2}-y^\star\|^2 +(\rho+\xi) \|x^{k+1} - x^{\star}\|^2\right)\\
&\geq
\rho {r^k}^2 + (\rho+\frac{\xi-\sigma}{2})n\sum_{i=1}^n \frac{1}{n} \left(\xplus - x^k_i\right)^2\\
&\quad- \rho n^2\left(\sum_{i=1}^n \frac{1}{n}\left(\xplus - x^k_i\right)\right)^2\\
&\quad+ \sigma \|{x}^{k} - x^\star\|^2 + \sigma \|{x}^{k+1} - x^\star\|^2.
\end{aligned}
\end{equation}
Therefore, using Jenzen's inequality, we can obtain
\begin{equation}
\label{xconverge}
\begin{aligned}
&\left(\frac{1}{\rho} \|y^{k+1}-y^\star\|^2 + (\rho+\xi)\| x^{k} - x^\star\|^2 \right)\\
&- \left(\frac{1}{\rho}\|y^{k+2}-y^\star\|^2 +(\rho+\xi) \|x^{k+1} - x^{\star}\|^2\right)\\
&\geq \rho {r^k}^2 + \sigma \|{x}^{k} - x^\star\|^2 + \sigma \|{x}^{k+1} - x^\star\|^2.
\end{aligned}
\end{equation}
Summing the inequality~\eqref{rk_ineq} from $k=1$ to $k=K$ and letting $K\to\infty$, we have
\begin{equation}
\label{xk_to}
\begin{aligned}
&\frac{1}{\rho} \|y^{1}-y^\star\|^2 + (\rho+\xi)\| x^{0}-x^\star\|^2\\
&\geq \sum_{k=1}^{\infty}\rho {r^k}^2 + \sum_{k=1}^{\infty}\left(\sigma \|{x}^{k} - x^\star\|^2 + \sigma \|{x}^{k+1} - x^\star\|^2\right).
\end{aligned}
\end{equation}
Thus $x^k\to x^\star$.

Lastly, we prove the convergence of the objective value. Because $x^k\to x^\star$ and $r^k\to 0$, all the terms on the right-hand sides of~\eqref{ppstar} and~\eqref{stronglyconvex} go to zeros. Thus, $p^k\to p^\star$.
\end{proof}

In the case when the objective function is quadratic, we can further show the convergence rate is linear. We omit the proof here due to the limitation of space.
\section{Numerical experiment}\label{sec:sim}
In the following, we refer to the algorithm proposed in~\cite{zhao2013optimal} as the ``dual algorithm.''

\subsection{Estimating $r^k$ from Local Frequency Measurements}\label{sec:est}
It was shown in~\cite{zhao2013optimal} that the estimator described in~\cite{kitanidis1987unbiased} may be used to estimate the primal residual $r^k$. This is achieved using a discrete LTI system model with $r^k$ as the input and system frequency $\omega^k$ as the output. Each load measures $\omega^k$ and essentially solves the LTI system to determine the value of $r^{k-1}$ to achieve that output. It was shown that the estimation error sequence forms a sequence with mean zero and bounded variance under certain conditions, which are satisfied in our simulations here. For more detailed information about the estimator and its properties, the reader is referred to~\cite{zhao2013optimal}.


\subsection{Simulation Setup}\label{sec:setup}
We consider the simulation scenario of~\cite{zhao2013optimal}, where there is a single generator in a micro-grid with $n$ loads. Figure~\ref{fig:setup} shows the simulation architecture. A disturbance $\bar{g}$ is applied to the system. Local generator controls (commonly used in generators) adjust generation output to reduce the consumption-generation mismatch---even without the presence of smart loads. A system-wide process disturbance, $\zeta$, and measurement noise, $\delta_i$, at each load are modeled as wide-sense stationary white noise. For more details about the simulation model (such as system and generator dynamics and noise statistics), the reader is referred to~\cite{zhao2013optimal}. For ease of comparison, we choose parameters as done in~\cite{zhao2013optimal}. For the sake of completeness, we summarize the simulation parameters below.

\begin{figure}[htpb]
\includegraphics[scale=0.25]{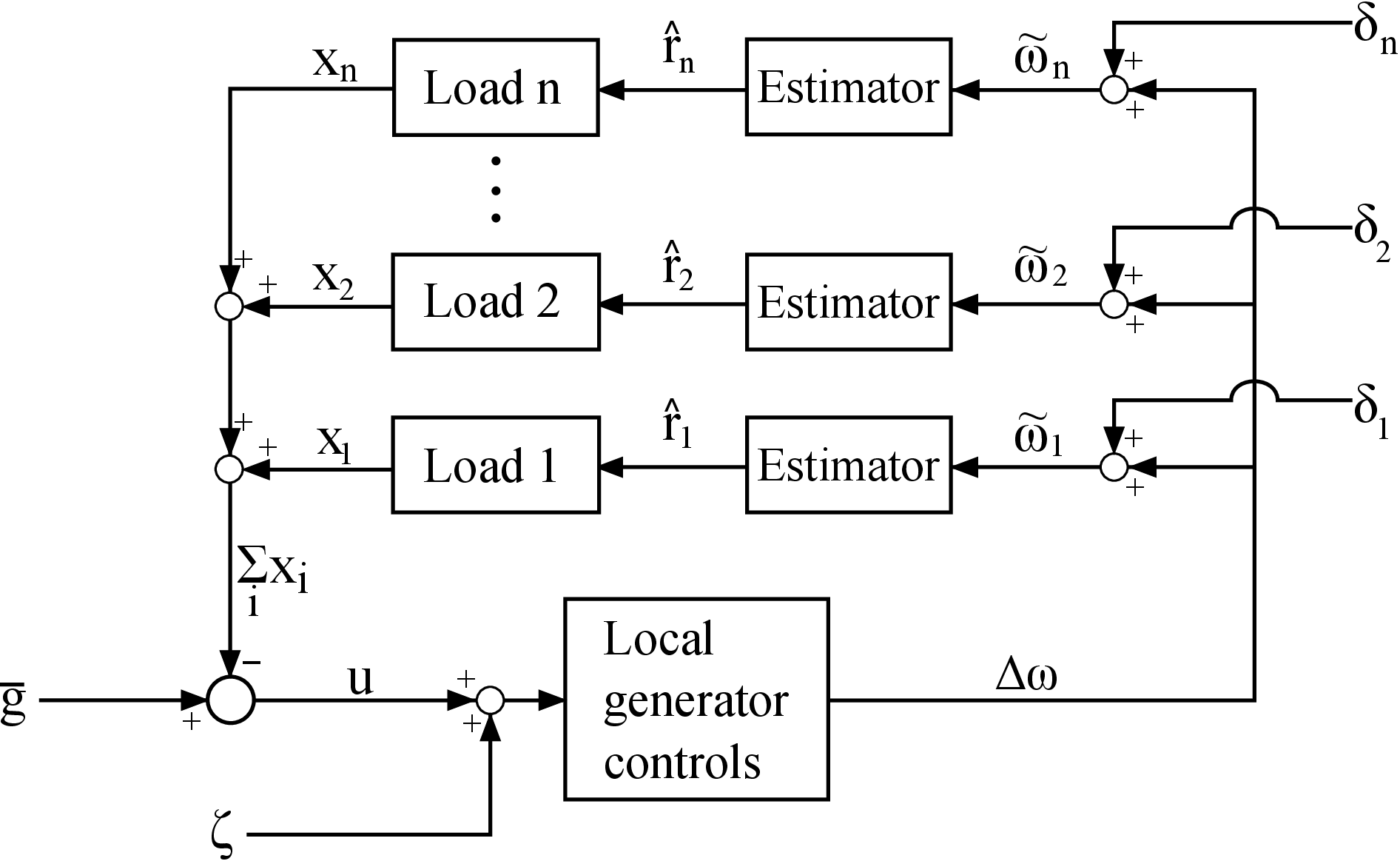}
\caption{Power system model architecture used for simulation.}
\label{fig:setup}
\end{figure}

Each load is constrained by the amount of it may vary consumption. We choose $a_i=0,~i=1,2,\ldots,n$, and we choose each $b_i$ from a uniform distribution; we then normalize the $b_i$'s so that $\sum_{i=1}^{n}b_i=60$~MW.

We use two models of consumer disutility: i) a non-continuously differentiable disutility, and ii) a quardatic disutility. The first disutility is modeled as
\begin{align}\label{eq:f1}
  f_i(x_i) =
  \begin{cases}
    q_i(x_i)^2, & |x_i| \leq \eta_i\\
    3q_i(x_i)^2-q_i\eta_i, & |x_i| \geq \eta_i,
  \end{cases}
\end{align}
For our simulations, we choose $\eta=0.1b_i$. The second disutility is modeled as
\begin{align}\label{eq:f2}
f_i(x_i)=\frac{1}{2}q_ix_i^2.
\end{align}
For each disutility model, $1/q_i$ is chosen from a uniform distribution on the interval [1, 3] for each $i$. 

The initial conditions of the system are $g^0=200$~MW and $x_i^0=0,~i=1,2,\ldots,n$. Two generation drops are modeled as step changes in generation:
\begin{equation*}
g^k = \left\{\begin{array}{l}
200~\mathrm{MW},\;\;~0~\mathrm{s} \leq kT < 20~\mathrm{s}\\
190~\mathrm{MW},~20~\mathrm{s} \leq kT < 50~\mathrm{s}\\
170~\mathrm{MW},~50~\mathrm{s} \leq kT, \end{array} \right.
\end{equation*}
where $T=0.1$ seconds is the discretization interval.

Although the DM-ADMM algorithm does not require communication among loads, communication may be used to average out some of the noise. In that case, the node set of the communication graph $\mathcal{G}$ is simply the set of loads $\mathcal{V}=\{1,2,\ldots,n\}$. Each load $i$ may communicate with other loads that share an edge with it. That is, load $i$ may communicate with all loads $j\in N(i)\triangleq\{j|(i,j)\in\mathcal{E}\}$, where $\mathcal{E}$ is the edge set of the $\mathcal{G}$. In this work, we report results both with and without communication among loads.

For the scenario with communication, we use a 1D-grid communication graph, where
\begin{align*}
N(i)=[\max\{1,i-n_0\},\ \min\{n,i+n_0\}],
\end{align*}
for a positive integer $n_0$. At each iteration $k$, agent $i$ receives the measurement $\hat{r}^k_j$ and value of $y^k_j$ from all its neighbors $j\in N(i)$. From those values, agent $i$ then computes the averaged values $\bar{y}_i^{k}$ and $\bar{r}^k_i$, which are then used in place of $y_i^k$ and $\hat{r}_i^k$ in~\eqref{update_i}, respectively.

Additionally, we choose $y_i^0=0$ for all $i$ and $\rho=2.5\times 10^{-3}$.

\subsection{Simulation Results}

\subsubsection{Non-continuously differentiable disutility}
Figure~\ref{grid1} shows the system frequency and consumer disutility for the proposed DM-ADMM algorithm for the disutility model~\eqref{eq:f2}. The dual algorithm is not implementable in this case. The scenarios both with and without estimation error are shown, and no communication is used in either scenario. The system frequency without smart loads is also shown in red.

It is clear from the figure that the DM-ADMM algorithm significantly reduces the frequency deviation from nominal during the contingency events. Although performance is better without estimation error, the DM-ADMM algorithm still significantly reduces the frequency deviation compared to the generator-only scenario. In the presence of estimation error, the consumer disutility is somewhat higher than when there is no error.

\begin{figure}[htpb]
\includegraphics[scale=0.325]{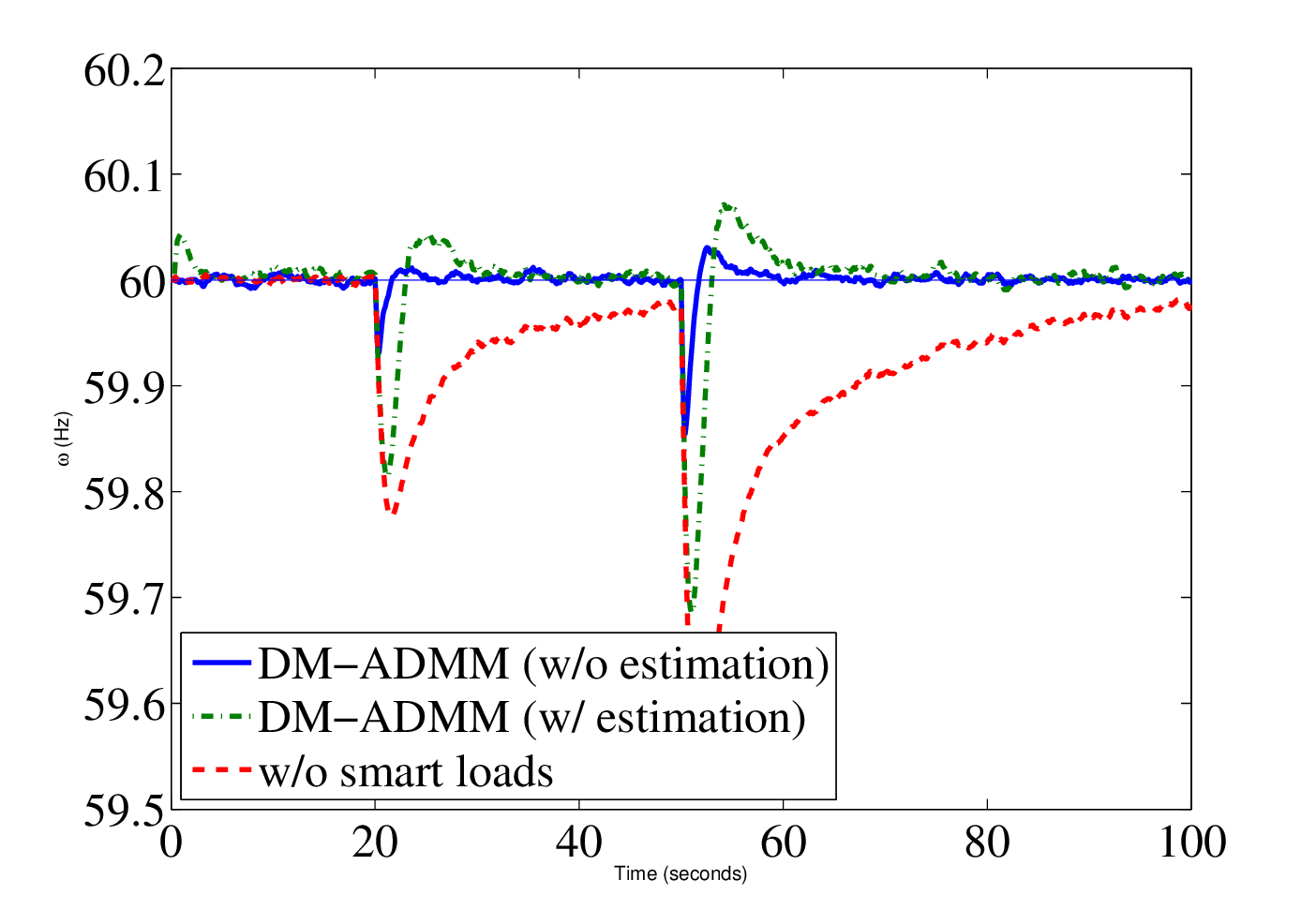}
\includegraphics[scale=0.325]{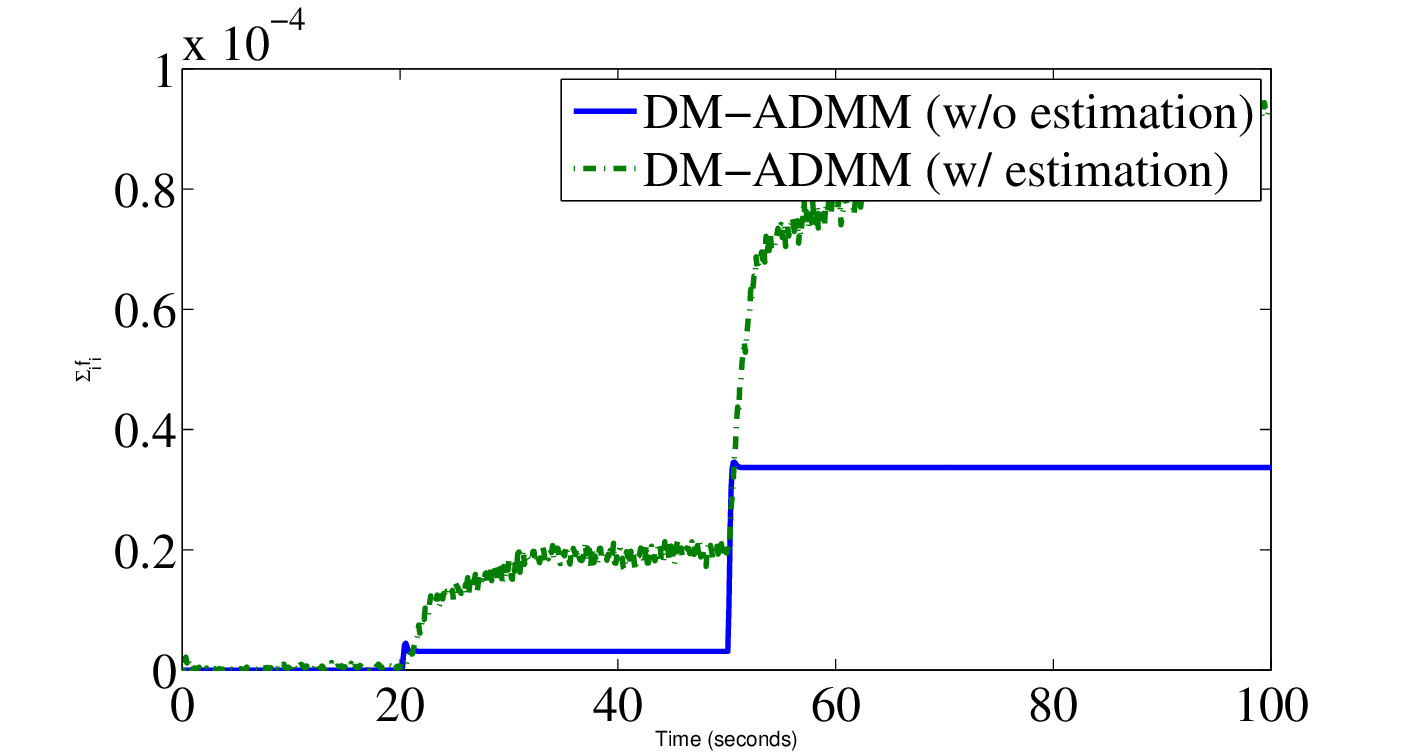}
\caption{System frequency and total consumer disutility for the proposed DM-ADMM algorithm when consumer disutilities are not continuously differentiable.}
\label{grid1}
\end{figure}

\subsubsection{Comparison with dual algorithm}
Figure~\ref{grid2} shows the system frequency and consumer disutility for both the DM-ADMM algorithm and the dual algorithm using disutility model~\eqref{eq:f2}. Both algorithms have estimation error and utilize communication among loads ($n_0$).

The DM-ADMM algorithm outperforms the dual algorithm in maintaining nominal system frequency. Although not reported here, increased communication was found to have little effect on each algorithm's ability to maintain frequency. 
The disutilities for the two algorithms are very similar. 

\begin{figure}[htpb]
\includegraphics[scale=0.325]{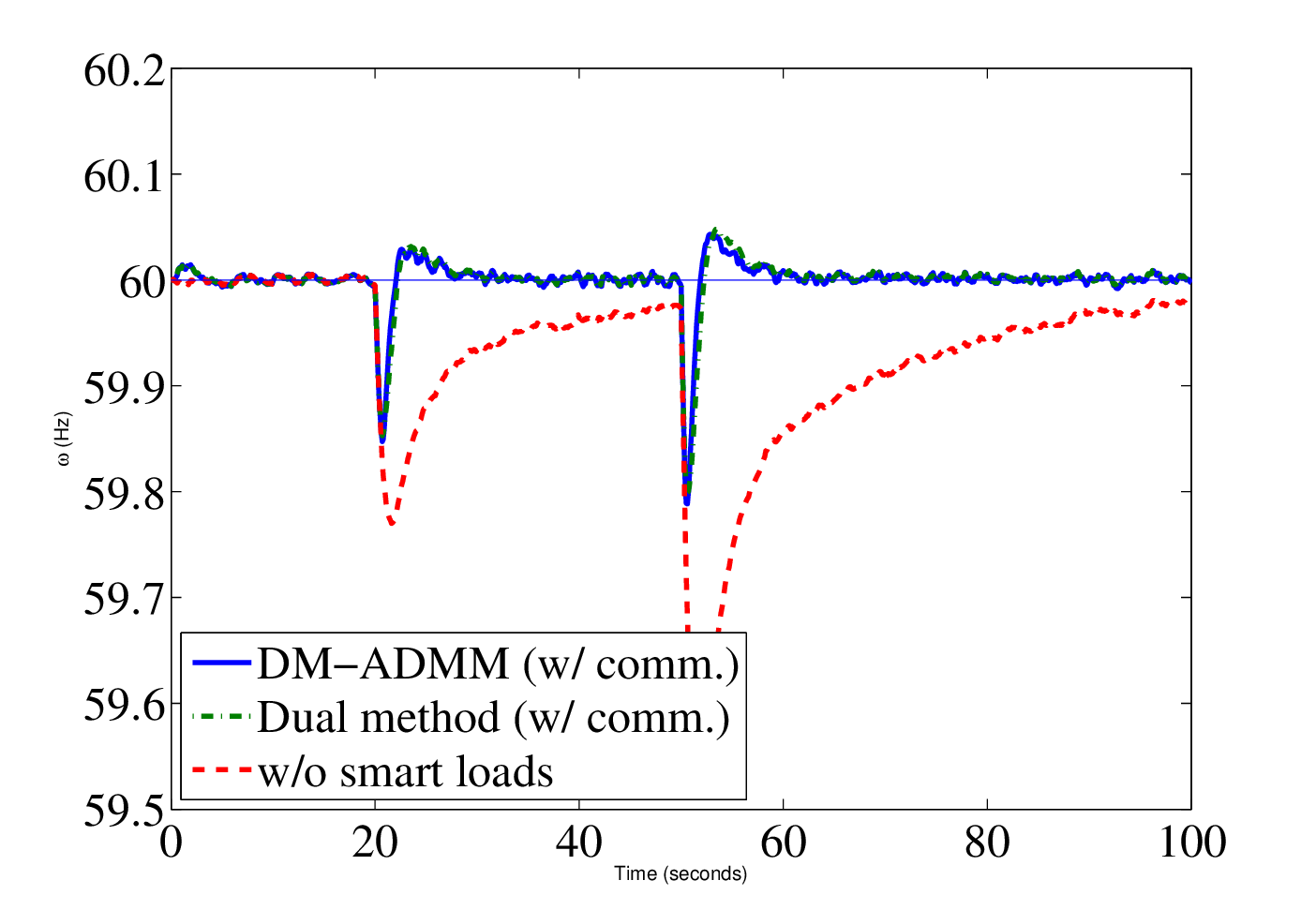}
\includegraphics[scale=0.325]{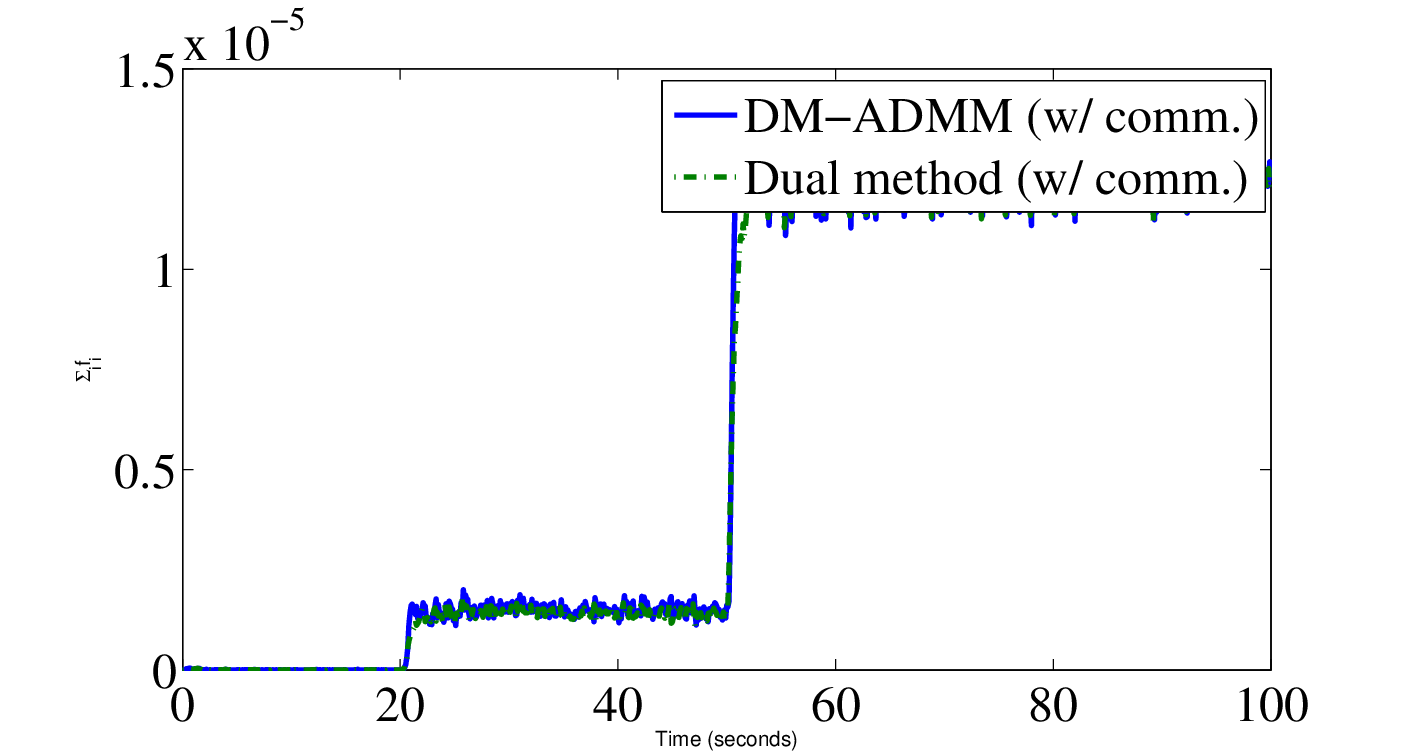}
\caption{System frequency and total consumer disutility for the proposed DM-ADMM algorithm and the dual algorithm with estimation error of consumption-generation mismatch.}
\label{grid2}
\end{figure}

\subsubsection{Comparison with PJ-ADMM}
Figure~\ref{grid3} shows the system frequency and consumer disutility for the DM-ADMM and PJ-ADMM algorithms using disutility model~\eqref{eq:f2}. There is no estimation error for each scenario. 
The PJ-ADMM algorithm has a few additional parameters compared to the DM-ADMM algorithm. We tuned these additional parameters and kept all other parameters the same.

As can be seen, the DM-ADMM algorithm has a smaller overshoot than does the PJ-ADMM algorithm as well as a smaller initial drop in frequency. That is, the PJ-ADMM algorithm does not have any improved performance, \emph{and} it has more parameters that require tuning---making it more difficult to implement in practice.

\begin{figure}[htpb]
\includegraphics[scale=0.325]{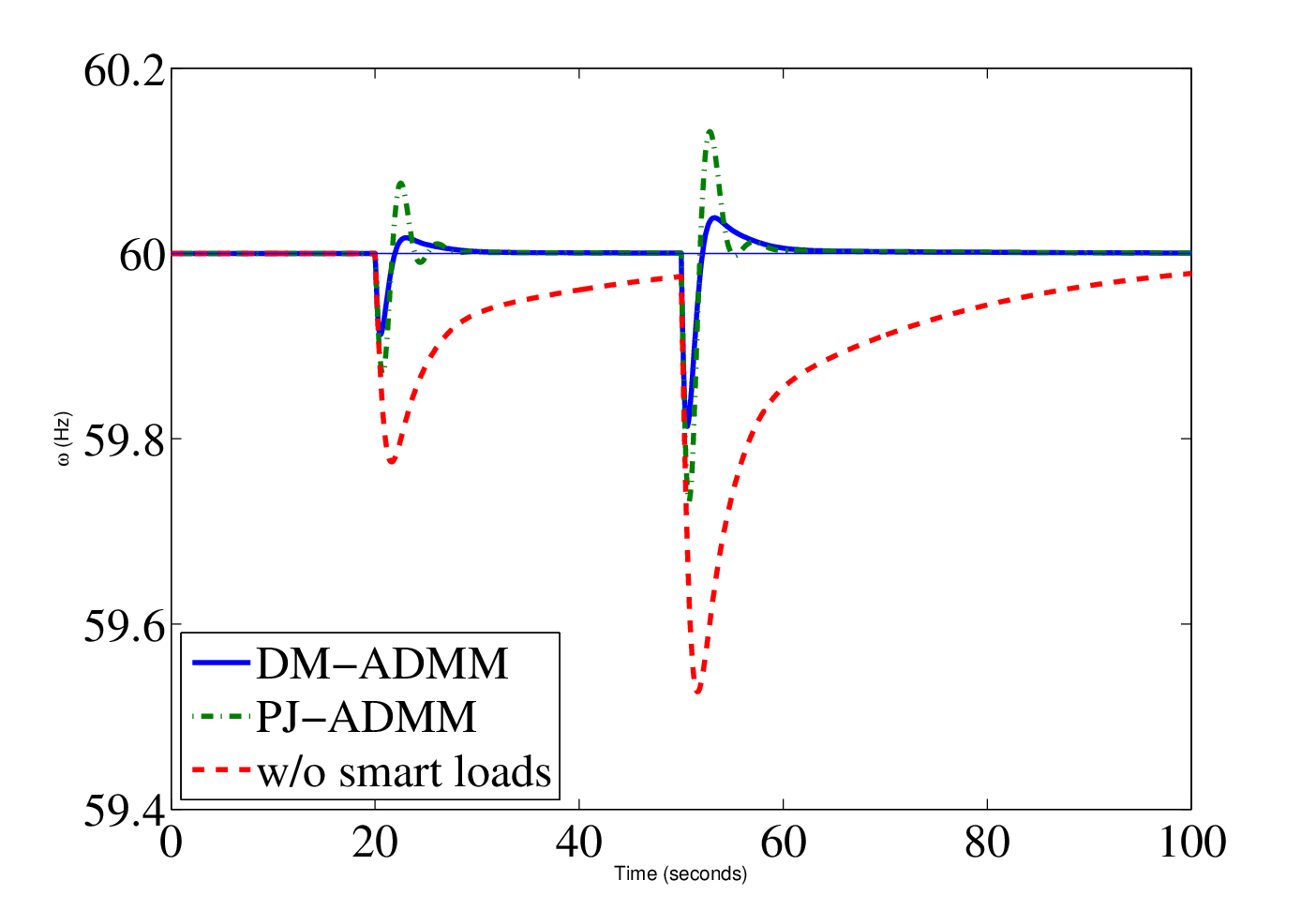}
\includegraphics[scale=0.325]{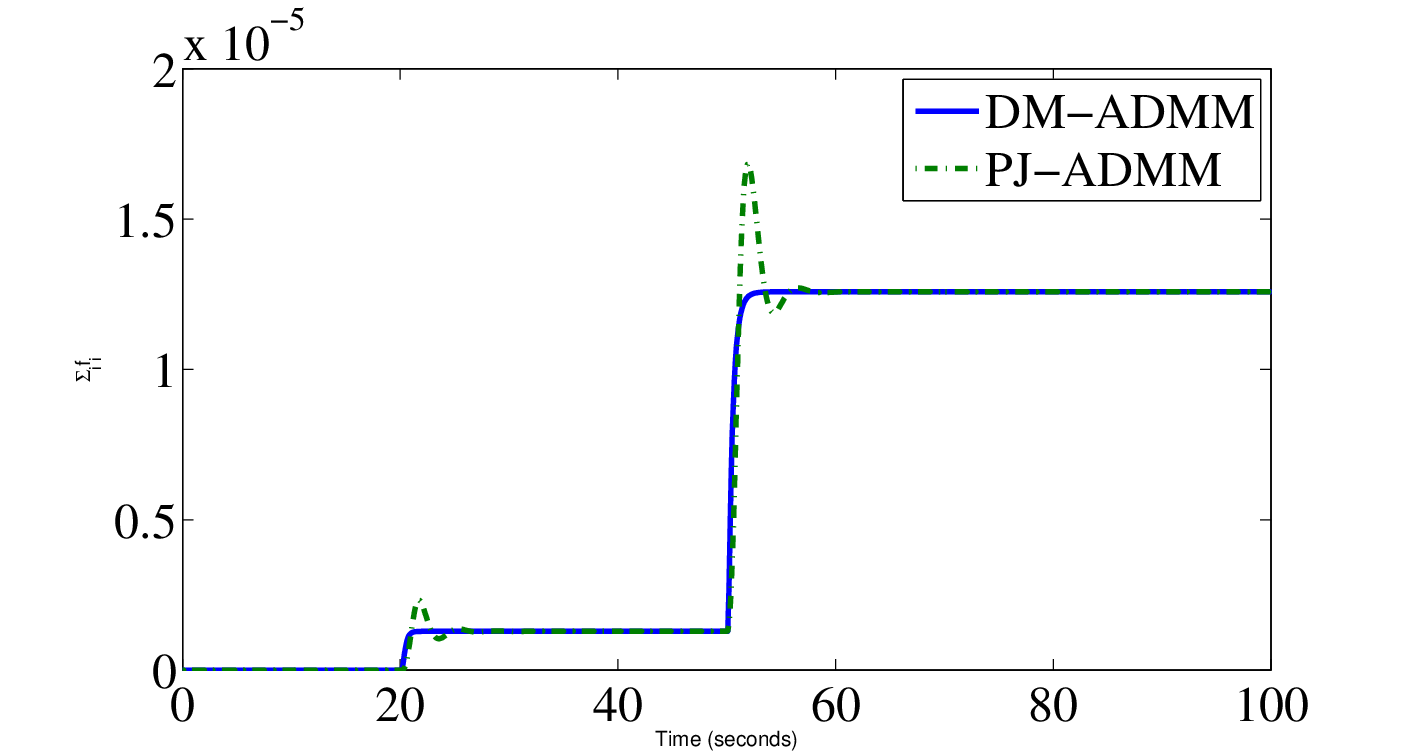}
\caption{System frequency and total consumer disutility for the proposed DM-ADMM algorithm and the PJ-ADMM algorithm without estimation error of consumption-generation mismatch.}
\label{grid3}
\end{figure}

\section{Conclusion}\label{sec:conc}
The proposed DM-ADMM algorithm solves a constrained optimization problem in a decentralized manner with application to using smart and flexible loads to assist generators in maintaining system frequency in the power grid. The objective of the algorithm is to change power consumption of loads to match generation while minimizing disutility associated with changes in consumption. Loads can estimate the mismatch between consumption and generation using local frequency measurements, and no communication with a central authority or other loads is required for each load to implement the DM-ADMM algorithm.

Previously, Zhao \emph{et al.}~\cite{zhao2013optimal} have solved the dual formulation of this problem. The main advantage of the DM-ADMM algorithm over the dual algorithm proposed by Zhao \emph{et al.} is that the DM-ADMM algorithm requires fewer restrictions on the consumer disutility functions. In particular, the dual algorithm requires the disutilities to be twice continuously differentiable with respect to changes in consumption. Conversely, disutilities for the DM-ADMM algorithm need not be continuously differentiable. This allows certain asymmetries in consumer disutilities that may lead to the disutility not being twice continuously differentiable and may be quite common.

A proximal-Jacobian ADMM (PJ-ADMM) algorithm was proposed in~\cite{deng2013parallel}, and it was supposed in~\cite{liu2015multi} that it might be used for application to smart grid. Here, we showed numerical results for the PJ-ADMM algorithm in a smart-grid application (no numerical results were reported in~\cite{liu2015multi}).

In this preliminary work, we proved loads using the DM-ADMM algorithm will converge to a feasible optimal point that minimizes consumer disutility while matching power consumption and power generation when the mismatch between the two is known exactly by each load. In practice, there will be some noise in the estimation of the mismatch from local frequency measurements. Simulations show that the DM-ADMM algorithm performs comparably to or better than the dual algorithm when applicable and that the DM-ADMM algorithm performs comparably to or better than the PJ-ADMM algorithm. The DM-ADMM algorithm is also more easily implemented than the PJ-ADMM algorithm because it has fewer parameters that must be tuned.

Future work will focus on analysis of the DM-ADMM algorithm for the stochastic case. In addition, future work will focus on relaxing the restriction of strong convexity on the consumer disutility. Another interesting avenue for future work is the extension of the DM-ADMM algorithm for time-varying generation (i.e., a changing equality constraint).

\section*{Acknowledgment}
The authors thank Changhong Zhao and Steven Low for their assistance in implementing the power grid simulation model used in this work.

\bibliographystyle{IEEEtran}
\bibliography{jj/library,jb/Jonathan}

\section*{Appendix}
\begin{proof}[Proof of Proposition~\ref{obj_cvgs}]
$\xplus$ is the minimizer in the update~\eqref{update_i}, we have
\[
0\in \partial f_i(\xplus) + {y_i^{k+1}} + \rho (\xplus+\sum_{j\neq i}x_j^{k}-C).
\]
Using the update rule $y^{k+2}  =  y^{k+1} + \rho (x_i^{k+1}+\sum_{j\neq i}x_j^{k+1}-C)$, we obtain
\begin{equation}
\label{zero_grad}
0\in \partial f_i(\xplus) + {y^{k+2}+\rho \big(\sum_{j\neq i}x_j^{k} - \sum_{j\neq i}x_j^{k+1}\big)}.
\end{equation}
From this inequality, we observe that $\xplus$ is the minimizer of
\begin{align*}
f_i(x) + \left( y^{k+2}+\rho (\sum_{j\neq i}x_j^{k} - \sum_{j\neq i}x_j^{k+1}) \right)x.
\end{align*}
Thus,
\begin{align}
\begin{aligned}
f_i(\xplus)&+\left( y^{k+2}+\rho (\sum_{j\neq i}x_j^{k} - \sum_{j\neq i}x_j^{k+1}) \right)^{k+1}\xplus\\
& \leq f_i(x_i^\star)+\left( y^{k+2}+\rho (\sum_{j\neq i}x_j^{k} - \sum_{j\neq i}x_j^{k+1}) \right)x_i^\star,
\end{aligned}
\end{align}
which implies
\begin{align}
\begin{aligned}
& f_i(\xplus) - f_i(x_i^\star)\\
&\leq \left( y^{k+2}+\rho \big(\sum_{j\neq i}x_j^{k} - \sum_{j\neq i}x_j^{k+1}\big) \right)(x_i^\star-\xplus).
\end{aligned}
\end{align}
Summing over all $i$'s, we obtain the result
\begin{align}
\begin{aligned}
p^{k+1} - p^\star \leq &-y^{k+2}r^{k+1} - \rho(r^k-r^{k+1})r^{k+1}\\
&- \rho\sum_{i=1}^n (x^k_i - x^{k+1}_i)(x_i^\star-\xplus)
\end{aligned}
\end{align}
\end{proof}

\begin{proof}[Proof of Lemma~\ref{lb_obj}]
Because $(x^\star, y^\star)$ is the global solution of the problem, it is also the minimizer of the unaugmented Lagrangian $L_0(x,y) = \sum_i f_i(x)+y(\sum_i x_i - C)$. $L_0(x,y^\star) = \sum_i f_i(x)+y^\star (\sum_i x_i - C)$, which is also strongly convex.
We have the inequality
\[
L_0(x^\star, y^\star) \leq L_0({x}^{k+1}, y^\star) - \xi \|{x}^{k+1} - x^\star\|^2
\]
Substituting in the expression of $L_0$, we obtain the result.
\end{proof}

\begin{proof}[Proof of Proposition~\ref{primal_res}]
From Lemma~\ref{lb_obj}, we have
\[
p^\star - p^{k+1}  \leq y^\star r^{k+1}-\xi\|x^{k+1}-x^*\|^2.
\]
Add this inequality to~\eqref{ppstar},
\begin{align*}
0  \leq &-(y^{k+2} - y^\star) r^{k+1} - \rho(r^k-r^{k+1})r^{k+1}\\
&- \rho\sum_{i=1}^n (x^k_i - x^{k+1}_i)(x_i^\star-\xplus) - \xi \|{x}^{k+1} - x^\star\|^2.
\end{align*}
Multiplying through by $-2$ yields
\begin{align}
\label{byneg2}
\begin{aligned}
0  \geq & 2(y^{k+2} - y^\star) r^{k+1} + 2\rho(r^k-r^{k+1})r^{k+1}\\
&+ 2\rho\sum_{i=1}^n (x^k_i - x^{k+1}_i)(x_i^\star-\xplus)+ 2\xi \|{x}^{k+1} - x^\star\|^2.
\end{aligned}
\end{align}
We examine the first term on the right-hand-side. Using the update rule $y^{k+2}  =  y^{k+1} + \rho r^{k+1}$, we have
\begin{equation}
\begin{aligned}
2(y^{k+2} - y^\star) r^{k+1} &= 2(y^{k+1} - y^\star) r^{k+1} + 2\rho \|r^{k+1}\|^2\\
&=\frac{2}{\rho}(y^{k+1} - y^\star) (y^{k+2} - y^{k+1})\\
&\quad +\frac{1}{\rho}\|y^{k+2} - y^{k+1}\|^2+ \rho \|r^{k+1}\|^2\\
&=\frac{1}{\rho}\left( \|y^{k+2}-y^\star\|^2-\|y^{k+1}-y^\star\|^2\right)\\
&\quad + \rho \|r^{k+1}\|^2.
\end{aligned}
\end{equation}
The last line above is obtained by using $y^{k+2} - y^{k+1} = y^{k+2} -y^\star + y^\star- y^{k+1}$. Then inequality~\eqref{byneg2} can be written as
\begin{align}
\label{primal_res_1}
\begin{aligned}
\frac{1}{\rho}&\left(\|y^{k+1}-y^\star\|^2-\|y^{k+2}-y^\star\|^2\right)\\
&\geq \rho {r^{k+1}}^2 + 2\rho(r^k-r^{k+1})r^{k+1}\\
&\quad + 2\rho\sum_{i=1}^n (x^k_i - x^{k+1}_i)(x_i^\star-\xplus) + 2\xi \|{x}^{k+1} - x^\star\|^2.
\end{aligned}
\end{align}
We first examine the last term above.
Using $x^k_i - x^{k+1}_i = (x^k_i -x_i^\star) - (x^{k+1}_i -x_i^\star)$ and vector notation, the last term becomes
\begin{equation}
\label{last2terms}
2\rho (x^k - x^{\star})\tr(x^\star-x^{k+1}) - 2\rho (x^{k+1} - x^{\star})\tr(x^\star-x^{k+1}).
\end{equation}
We add and subtract to this term the quantity
\begin{align*}
\rho\| x^{k+1} - x^k\|^2 = &\rho\| x^{k+1} - x^\star\|^2 + \rho\| x^{k} - x^\star\|^2\\
&- 2\rho(x^{k+1} - x^\star)\tr (x^{k} - x^\star).
\end{align*}
After rearranging, we rewrite~\eqref{last2terms} as
\begin{equation}
\rho\| x^{k+1} - x^k\|^2 - \rho\| x^{k} - x^\star\|^2
+\rho \|x^{k+1} - x^{\star}\|^2.
\end{equation}
Now inequality~\eqref{primal_res_1} becomes
\begin{align}
\begin{aligned}
&\left(\frac{1}{\rho} \|y^{k+1}-y^\star\|^2 + \rho\| x^{k} - x^\star\|^2 \right)\\
&- \left(\frac{1}{\rho}\|y^{k+2}-y^\star\|^2 +\rho \|x^{k+1} - x^{\star}\|^2\right)\\
&\geq \rho {r^{k+1}}^2 + 2\rho(r^k-r^{k+1})r^{k+1} + \rho\| x^{k+1} - x^k\|^2\\
&\quad + 2\xi \|{x}^{k+1} - x^\star\|^2.
\end{aligned}
\label{primal_res_almost}
\end{align}
The first two terms on the right-hand-side can be rewritten as
\[
\rho {r^k}^2 - \rho (r^{k+1}-r^k)^2.
\]
Adding $\xi \|{x}^{k} - x^\star\|^2 - \xi \|{x}^{k+1} - x^\star\|^2$ on both sides, we then obtain
\begin{align}
\begin{aligned}
&\left(\frac{1}{\rho} \|y^{k+1}-y^\star\|^2 + (\rho+\xi)\| x^{k} - x^\star\|^2 \right)\\
&- \left(\frac{1}{\rho}\|y^{k+2}-y^\star\|^2 +(\rho+\xi) \|x^{k+1} - x^{\star}\|^2\right)\\
&\geq \rho {r^k}^2 + \rho\| x^{k+1} - x^k\|^2 - \rho (r^{k+1}-r^k)^2\\
&\quad + \xi \|{x}^{k} - x^\star\|^2 + \xi \|{x}^{k+1} - x^\star\|^2.
\end{aligned}
\end{align}
We rewrite the last two terms using the inequality
\begin{equation}
\begin{aligned}
\xi \|{x}^{k} - x^\star\|^2 + \xi \|{x}^{k+1} - x^\star\|^2\geq\frac{\xi}{2} \|{x}^{k+1} - {x}^{k}\|^2.
\end{aligned}
\end{equation}
Using 
\begin{equation*}
\begin{aligned}
(r^{k+1}-r^k)^2 &=& & \left(\sum_{i=1}^n \xplus -\sum_{i=1}^n x^k_i\right)^2\\
&=& & \left(\sum_{i=1}^n \left(\xplus - x^k_i\right)\right)^2
\end{aligned}
\end{equation*}
and
\[
\| x^{k+1} - x^k\|^2 = \sum_{i=1}^n  \left(\xplus - x^k_i\right)^2,
\]
we can re-write the inequality~\eqref{primal_res_almost} as
\begin{align}
\label{jenzen}
\begin{aligned}
&\left(\frac{1}{\rho} \|y^{k+1}-y^\star\|^2 + (\rho+\xi)\| x^{k} - x^\star\|^2 \right)\\
&- \left(\frac{1}{\rho}\|y^{k+2}-y^\star\|^2 +(\rho+\xi) \|x^{k+1} - x^{\star}\|^2\right)\\
&\geq \rho {r^k}^2 + (\rho+\frac{\xi}{2})n\sum_{i=1}^n \frac{1}{n} \left(\xplus - x^k_i\right)^2\\
&\quad - \rho n^2\left(\sum_{i=1}^n \frac{1}{n}\left(\xplus - x^k_i\right)\right)^2.
\end{aligned}
\end{align}
Finally, using the condition $\displaystyle\rho\leq \frac{\xi}{2\left(n-1\right)}$ and applying Jenzen's inequality to the last two terms, we have
\begin{equation*}
\begin{aligned}
(\rho+\frac{\xi}{2})n\sum_{i=1}^n \frac{1}{n} \left(\xplus - x^k_i\right)^2
- \rho n^2\left(\sum_{i=1}^n \frac{1}{n}\left(\xplus - x^k_i\right)\right)^2\\
\geq 0.
\end{aligned}
\end{equation*}
The result follows.
\end{proof}
\end{document}